\documentclass[12pt]{amsart}
\usepackage{amsmath}
\usepackage{amsfonts}
\usepackage{amsthm}
\usepackage{amssymb}
\usepackage{amscd}
\usepackage[all]{xy}
\usepackage{enumerate}


\theoremstyle{plain}
\newtheorem{theorem}{Theorem}[subsection]

\newtheorem{lemma}[theorem]{Lemma}

\theoremstyle{definition}
\newtheorem{definition}[theorem]{Definition}

\newtheorem{remark}[theorem]{Remark}

\usepackage{color}

\newcommand{\Z}{\mathbb{Z}}

\newcommand{\C}{\mathbb{C}}

\newcommand{\OC}{\mathcal{O}}
\newcommand{\PR}{\mathbb{P}}

\newcommand{\Cu}{\mathcal{C}} 
\newcommand{\Du}{\mathcal{D}} 
\newcommand{\Surf}{\mathcal{S}}  
\newcommand{\Jac}{\mathcal{J}} 
\newcommand{\GAU}{\mathit{G}} 
\newcommand{\GrpJ}{\mathcal{G}} 
\newcommand{\addl}{\oplus}  
\newcommand{\addlS}{\oplus_X} 
\newcommand{\addlSP}[1]{\addl^{#1}} 
\newcommand{\PUNKT}{\ \ \text{.}}  
\newcommand{\BEISTRICH}{\ \ \text{,}}  

\def\daschmapsto{\mathrel{\mapstochar\dashrightarrow}} 

\newcommand{\bigslant}[2]{{\raisebox{.2em}{$#1$}\left/\raisebox{-.2em}{$#2$}\right.}}  

\numberwithin{equation}{section}

\newcommand{\beba}  {\begin{equation}\begin{array}{rcl}}

\newcommand{\eaee}  {\end{array}\end{equation}}

\makeatletter
\def\l@section{\@tocline{1}{0pt}{1pc}{}{}}
\def\l@subsection{\@tocline{2}{0pt}{1pc}{4.6em}{}}
\def\l@subsubsection{\@tocline{3}{0pt}{1pc}{7.6em}{}}
\renewcommand{\tocsection}[3]{%
  \indentlabel{\@ifnotempty{#2}{\makebox[2.3em][l]{%
    \ignorespaces#1 #2.\hfill}}}#3}
\renewcommand{\tocsubsection}[3]{%
  \indentlabel{\@ifnotempty{#2}{\hspace*{2.3em}\makebox[2.3em][l]{%
    \ignorespaces#1 #2.\hfill}}}#3}
\renewcommand{\tocsubsubsection}[3]{%
  \indentlabel{\@ifnotempty{#2}{\hspace*{4.6em}\makebox[3em][l]{%
    \ignorespaces#1 #2.\hfill}}}#3}
\makeatother

\title[Biquadratic addition laws and canonical map of $(1,2,2)$-Theta divisor]{Biquadratic addition laws on elliptic curves in $\PR^3$ and the canonical map of the $(1,2,2)$-Theta divisor}

\thanks{The present work was supported by the ERC Advanced grant n. 340258, TADMICANT.}

\author{Luca Cesarano}

\address{Luca Cesarano \newline Lehrstuhl Mathematik VIII, Universit\"at Bayreuth \newline
Universit\"atsstra\ss e 30, 95447, Bayreuth, Germany \newline
\texttt{luca.cesarano@uni-bayreuth.de}}

\begin{document}

\begin{abstract}
We recall that a smooth ample surface $\Surf$ in a general $(1,2,2)$-polarized abelian threefold, 
which is the pullback of the Theta divisor of a smooth plane quartic curve $\Du$, is a surface
isogenous to the product $\Cu \times \Cu$, where $\Cu$ is a genus $9$ curve embedded in $\PR^3$ 
as complete intersection of a smooth quadric and a smooth quartic. We show that the space of 
global holomorhic sections of the canonical bundle of this surface is generated by certain 
determinantal bihomogeneous polynomials of bidegree $(2,2)$ on $\PR^3$, which can be used to 
define biquadratic addition laws on the Jacobi model of elliptic curves, embedded in $\PR^3$ as 
complete intersection of two quadrics. Finally, we use this interesting relationship with the 
biquadratic addition laws to describe the behavior of the canonical map of $\Surf$.
\end{abstract}
\maketitle

\section{Introduction}
Let $(A, \mathcal{L})$ be a general $(1,2,2)$-polarized abelian threefold. We can consider an isogeny $p$ onto a principally polarized abelian threefold $(\Jac, \Theta)$, and we denote its kernel by $\GrpJ$, which is a group isomorphic to $\Z^2$ acting by translations on $A$. 
By our generality assumption on $A$, we can assume that $\Jac$ is the Jacobian variety of a non-hyperelliptic quartic plane $\Du$. Once identified $\Du$ with its embedded image in $\Jac$ through the Abel-Jacobi map, we can consider the pullback of $\Du$ through $p$, which we denote by $\Cu$. The curve $\Cu$ is a smooth projective curve of genus $9$ with an unramified bidouble cover $\pi:=p|_{\Cu}$ onto $\Du$.
It is well-known (see for instance \cite{ACGH} p. 226), that the Theta divisor $\Theta$ of $\Du$ is a translated with the vector of Riemann constants theta-characteristic of the subvariety 
\begin{equation*}
 W_2(\Du) = \{\mathcal{L} \in Pic^2(\Du) \ : h^0(\Du, \mathcal{L}) \geq 1 \} \PUNKT
\end{equation*}
According to Riemann's Singularity Theorem (cf. \cite{ACGH} p.226), $W_2(\Du)$ is singular precisely when $\Du$ is hyperelliptic. Hence, by our generality
assumption on $A$ and on $\Du$, we can identify $W_2(\Du)$ with the second symmetric product $\Du^{(2)}$, the latter defined as the quotient of the product $\Du \times \Du$ by the natural involution on the two factors. 
In particular, $\Du^{(2)}$ is a smooth surface which we regard as the set of effective divisors of degree $2$ on $\Du$. \newline

In this paper, we are interested in the problem of a purely geometrical description of the canonical map of the surface $\Surf$ obtained by pulling back the Theta divisor $\Theta$ to $A$ through $p$. From the definition of $\Surf$, it follows that $\Surf$ is a bidouble unramified cover, which we can geometrically describe as a quotient
 

\begin{equation} \label{RepresentS}
\Surf = \bigslant{\Cu \times \Cu}{\Delta_{\GrpJ} \times \Z_2} \BEISTRICH
\end{equation}
where 
 $\Delta_{\GrpJ}$ denotes the diagonal subgroup of $\GrpJ \times \GrpJ$, acting naturally on the factors of the product $\Cu \times \Cu$, and where $\Z_2$ acts naturally on the two factors. \newline

In the previous work \cite{Cesarano2018}, we studied the canonical map of $\Surf$ by using projection methods: the isogeny $p$ factors through each of the three projections onto the three $(1,1,2)$-polarized abelian threefold, each of them obtained as a quotient of $A$ by a non-trivial element of $\GrpJ$. Therefore, it is possible to investigate the behavior of canonical map of $\Surf$ by looking closely at the canonical map of the corresponding quotients, which are surfaces of type $(1,1,2)$. The canonical map is of degree $2$ and factors through a regular surface with $32$ nodes, with $p_g=4$ and $K^2 = 6$ (see \cite{CataneseSchreyer1}).\newline
In this paper, we aim to study the canonical map of $\Surf$ by using only the presentation \ref{RepresentS} of $\Surf$, which is, in particular, a surface isogenous to the product $\Cu^2$. We recall that surfaces isogenous to a product are the quotient of the form $\bigslant{C_1 \times C_2}{G}$, where $C_1$ and $C_2$ are smooth projective curves and $G$ is a finite group acting freely of $C_1 \times C_2$. However, a satisfactory geometrical description of the canonical map of such surfaces has turned out to be in general a very challenging task (see \cite{Catanese2018}), which we leave aside in the hope to be able to address some aspects of this question in future. In our concrete case, it turned out that there exists a relationship between the canonical map of the previously defined surface $\Surf$  and some bihomogeneous polynomials of bidegree $(2,2)$ in four variables which define addiction laws on certain elliptic curves in $\PR_3$. To introduce and describe more precisely this relationship, we start with the following lemma, which provides to us a very useful representation of the curve $\Cu$ and the unramified bidouble cover $p \colon \Cu \longrightarrow \Du$ (cf. \cite{Cesarano2018}, lemma 2.6).
\begin{lemma} \label{fundamentallemmagenus9}
Let $(A, \mathcal{L})$ be a general $(1,2,2)$-polarized Abelian 3-fold, let $p: A \longrightarrow \Jac$ be an isogeny onto the
Jacobian of a general algebraic curve $\Du$ of genus $3$.
Let us moreover consider the algebraic curve $\Cu$ obtained by pulling back to $A$
the curve $\Du$, according to the following diagram: 
\begin{equation} \label{pullbackdiag12}
\xymatrix { \Cu \ar@{^{(}->}[r]\ar[d]_{}& A\ar[d]_{} \\
\Du \ar@{^{(}->}[r]& \Jac }
\end{equation}
Then, the following hold true:
\begin{itemize}
\item The genus $9$ curve $\Cu$ admits $\mathcal{E}$ and $\mathcal{F}$ two distinct $\GrpJ$-invariant $g_{4}^1$'s, with 
$\mathcal{E}^{2} \ncong \mathcal{F}^{2}$ and
\begin{equation*}
 h^0(\Cu, \mathcal{E}) = h^0(\Cu, \mathcal{F}) = 2 \PUNKT
\end{equation*}
\item The line bundle $\mathcal{M} := \mathcal{E} \otimes \mathcal{F}$ is a very ample theta characteristic of type $g_8^3$.
\item The image of $\Cu$ in $\PR^3 = \PR(\mathcal{M})$ is a complete intersection of the following type:
\begin{equation} \label{corollaryscrolleq}
\Cu : \begin{cases}
X^2 + Y^2 + Z^2 + T^2 = 0 \\
q(X^2,Y^2,Z^2,T^2) = XYZT \BEISTRICH
\end{cases}
\end{equation}
where $q$ is a quadric, and there exist coordinates $[X,Y,Z,T]$ on $\PR^3$ and two generators $a$, $b$ of $\GrpJ$ 
such that the projective representation of $\GrpJ$ on $\PR^3$ is represented by
\begin{align}\label{actionofgenerators}
\begin{split}
a.[X,Y,Z,T] &= [X,Y,-Z,-T]\\
b.[X,Y,Z,T] &= [X,-Y,Z,-T] \PUNKT
\end{split}
\end{align} 
\item The unramified covering $p: \Cu \longrightarrow \Du$ can be expressed as the map obtained by restricting to 
$\Cu$ the rational map $\psi: \PR^3 \dashrightarrow \PR^3$ defined by
\begin{equation} \label{corollarycover}
\psi: [X,Y,Z,T] \daschmapsto [x,y,z,t] := [X^2,Y^2,Z^2,T^2]  \PUNKT
\end{equation}
and the equations of $\Du$ in $\PR^3 = \PR[x,y,z,t]$ are, according with \ref{corollaryscrolleq}, in the following form:
\begin{equation} \label{corollaryscrolleq}
\Du : \begin{cases}
x + y + z + t = 0 \\
q(x,y,z,t)^2 = xyzt \PUNKT
\end{cases}
\end{equation}
\end{itemize}
\end{lemma}

We represent the points of $\mathcal{S}$ by equivalence classes $[(P,Q)]$ of points of $\Cu \times \Cu$ in $\PR^3 \times \PR^3$ with coordinates $[X_0 \cdots T_0]$ and $[X_1, \cdots T_1]$ on the two factors. Since the natural action of $\GrpJ$ by translations in $A$ is equivalent to the action of $\GrpJ$ on the cosets of the diagonal subgroup in $\GrpJ^2$, the action of $\GrpJ$ on $\Surf$ can be naturally represented as the action of $\GrpJ$ on the second component of $\Cu^2$: if $g$ denotes an element of $\GrpJ$ and $[(P,Q)]$ a point on $\Surf$, we have 
\noindent 
\begin{equation*}  
g.[P,Q] =[P,g.Q]  \PUNKT
\end{equation*}


\noindent We can now easily exhibit a basis for $H^0(\Surf, \omega_{\Surf})$. By lemma \ref{fundamentallemmagenus9}, $\omega_{\Cu}$ is the restriction of $\OC_{\PR}(2)$ on $\Cu$ and $H^0(\Cu, \omega_{\Cu})$ splits into a direct sum of $\GrpJ$ invariant subspaces 
\begin{align*}
H^0(\Cu, \omega_{\Cu}) = W_{++} \oplus W_{+-} \oplus W_{-+} 
\oplus W_{--} \BEISTRICH
\end{align*}
according to the signs of the action of two fixed generators $a$ and $b$ of $\GrpJ$ on the coordinates of $\PR^3$ (see \ref{actionofgenerators}).
\begin{align*}
W_{++} &= \left< X^2|_{\Cu},  Y^2|_{\Cu},  Z^2|_{\Cu} \right> \\
W_{+-} &= \left<XY|_{\Cu},  ZT|_{\Cu} \right> \\
W_{-+} &= \left<XZ|_{\Cu},  YT|_{\Cu} \right> \\
W_{--} &= \left<XT|_{\Cu},  YZ|_{\Cu} \right> 
\end{align*}
Recall that $\Surf$ is defines as the quotient of $\Cu \times \Cu$ by the action of $\Delta_G \times \Z_2$. Hence, 
\begin{equation*}
H^0(\Surf, \omega_{\Surf}) = H^0(\Cu \times \Cu, \omega_{\Cu} \boxtimes \omega_{\Cu} )^{\Delta_G \times \Z_2}
\end{equation*}
where the latter vector space denotes the $\Delta_G \times \Z_2$-invariant global holomorphic sections of $\omega_{\Cu} \boxtimes \omega_{\Cu}$. 

\noindent In conclusion, the following is a basis for $H^0(\Surf, \omega_{\Surf})$:
 \begin{equation} \label{polynomialexpressioncanonicalmap}
 \begin{split}
\eta_{12} &:= \begin{vmatrix} 					X_0^2	&	X_1^2\\
                         Y_0^2	 &	 Y_1^2
 					 \end{vmatrix} \\
\eta_{13} &:=  \begin{vmatrix} 					X_0^2	&	X_1^2\\
  																Z_0^2	&	Z_1^2
 					 \end{vmatrix} \\		
\eta_{23} &:=  \begin{vmatrix} 					Y_0^2	&	Y_1^2\\
  																Z_0^2	&	Z_1^2
 					 \end{vmatrix} \\		 			 
\omega_{45} &:= 	
					\begin{vmatrix} 									X_0Y_0	&	X_1Y_1\\
  																Z_0T_0	&	 Z_1T_1\\
 					 \end{vmatrix} \\		
\omega_{67}  &:=
					 \begin{vmatrix} 								X_0Z_0	&	X_1Z_1\\
  																Y_0T_0	&	 Y_1T_1\\
 					 \end{vmatrix} \\			 
\omega_{89} &:=  
					\begin{vmatrix} 									X_0T_0	&	X_1T_1\\
  																Y_0Z_0	&	 Y_1Z_1\\
 					 \end{vmatrix}  \PUNKT
\end{split}
\end{equation}
These determinantal polynomials of bidegree $(2,2)$ turn out to be strictly related to some biquadratic addition laws on the elliptic curves in $\PR^3$ defined as follows: we fix two non-zero complex numbers $u$ and $v$ such that $u+v$ is also non-zero. Then the following complete intersection is a smooth elliptic curve in $\PR^3$: 
\begin{align} \label{EquationsJacobiINTRO}
\mathcal{J}_{u,v}: \begin{cases}
uX^2 + Y^2 &= Z^2\\
vX^2 + Z^2 &= T^2 
\end{cases} 
\end{align}
One can show that (cf. \cite{AdditionLaw}, p.22) that the rational map $\oplus: \PR^3 \times \PR^3 \dashrightarrow \PR^3$ defined by the four biquadratic polynomials $[X_0^2Y_1^2 - Y_0^2X_1^2, X_0Y_0Z_1T_1 - Z_0T_0X_1Y_1, 
X_0Z_0Y_1T_1 - Y_0T_0X_1Z_1, X_0T_0Y_1Z_1 - Y_0Z_0X_1T_1]$ coincides, whenever defined, with the group law on $\mathcal{J}_{u,v}$. If we compare these polynomials with
$\eta_{01}$, $\omega_{45}$, $\omega_{67}$, $\omega_{89}$ in \ref{polynomialexpressioncanonicalmap}, we clearly notice that $\oplus|_{\Cu \times \Cu} = [\eta_{01}, \omega_{45}, \omega_{67}, \omega_{89}]$. \newline 

The paper is organized as follows. We recall in section $2$ the general definition of addition law on a fixed, embedded abelian variety, and we specialize it to the case of biquadratic addition laws on embedded elliptic curves in $\PR^3$. In the last section, we show how the previously mentioned relationship with the biquadratic addition laws on embedded elliptic curves of the form $\mathcal{J}_{u,v}$ can be used to study the canonical map of $\Surf$. We recall that, by the projection formula, it holds a decomposition in $1$-dimensional vector spaces
\begin{align*}
H^0(A, \mathcal{L}) &= \bigoplus_{\chi \in \GrpJ}H^0(\Jac, \OC_{\Jac}(\Theta) \otimes L_{\chi}) \BEISTRICH
\end{align*}
where $L_{\chi}$ is a $2$-torsion line bundle on $\Jac$.
Clearly, for every non-trivial element $g$ there exists a unique non-trivial character $\chi$ of $\GrpJ$ whose kernel is generated by $g$, and we can define $\Surf_g$ as the zero locus of a non-zero section of $H^0(\Jac, \OC_{\Jac}(\Theta) \otimes L_{\chi})$. We conclude the third section with a proof of the following result:

\begin{theorem} \label{FUNDAMENTALPULLBACKINTRO}
Let $U$, $V$ be points on $\Surf$ such that $\phi_{\Surf}(U) = \phi_{\Surf}(V)$. Then one of the following cases occurs:
\begin{itemize}
\item $V = U$
\item $V = -g.U$ for some non-trivial element $g$ of $\GrpJ$. This case arises precisely when $U$ and $V$ belong to the canonical curve $\Surf \cap \Surf_g$.   
\item $V = g.U$  for some non-trivial element $g$ of $\GrpJ$.  This case arises precisely when $U$ and $V$ belong to the translate $\Surf_{h}$, for every $h \in \GrpJ - \{g\}$.
\item $U$ and $V$ are two base points of $|\Surf|$ which belong to the same $\GrpJ$-orbit.
\end{itemize}
\end{theorem}



\section{Addition laws of bidegree $(2,2)$ on elliptic curves in $\PR^3$}
\noindent Throughout this section, we denote by $(A, \mathcal{L})$ a polarized abelian variety, with $\mathcal{L}$ a very ample line bundle. We denote by $\phi$ the holomorphic embedding $\phi_{|\mathcal{L}|} \colon A \longrightarrow \PR^N$ defined by the linear system $|\mathcal{L}|$ on $A$. 
Moreover, we denote by $\mu, \delta \colon A \times A \longrightarrow A$ the morphisms respectively defined by the sum and the difference in $A$, and by $\pi_1$ and $\pi_2$ the projections of $A \times A$ onto the respective factors. To prevent misunderstandings, we refer to $\mu$ as the \textbf{group law} on $A$, and we distinguish it from the notion of addition law which we are going to introduce in this section. An \textbf{addition law} $\addl$ of bidegree $(m,n)$ is a rational map $\PR^N \times \PR^N \dashrightarrow \PR^N$ defined by an ordered set of bihomogeneous polynomials $(f_0, \cdots, f_N)$, each of bidegree $(m,n)$, and such that there exists a non-empty Zariski open set $U$ of $A \times A$ on which $\addl$ and $\mu$ coincide (see also \cite{AdditionLaw}).

\noindent An addition law on $A$ can be viewed then as a rational map $\addl \colon \PR^N \times \PR^N \dashrightarrow \PR^N$ such that the following diagram commutes: 
\begin{equation} \label{additionlawdiagram}
\xymatrix { A \times A \ar[r]_{\phi \times \phi}\ar[d]_{\mu}& \PR^N \times \PR^N \ar@{-->}[d]_{\addl} \\
A \ar[r]_{\phi}& \PR^N }
\end{equation}
\noindent Assigned an addition law $\addl$ on $A$ defined by bihomogeneous polynomials $(f_0 \cdots f_N)$, we denote by $W(\addl)$ the sublinear system of $|\mathcal{O}_{\PR^N}(m) \boxtimes \mathcal{O}_{\PR^N}(n)|$ generated by $f_0, \cdots, f_N$.

\noindent In particular, the rational map $A \times A \dashrightarrow \PR^N$, which in diagram \ref{additionlawdiagram} is defined as the composition $\phi \times \phi$ with $\addl$, is defined by $N+1$ linearly independent global sections of
\begin{align*}
(\phi \times \phi)^*\left(\mathcal{O}_{\PR^N}(m) \boxtimes \mathcal{O}_{\PR^N}(n)\right) = \pi_1^* \mathcal{L}^{m} \otimes \pi_2^* \mathcal{L}^{n} \PUNKT
\end{align*}
The morphism $\phi \circ \mu$ is defined, on the other side, by the complete linear system $|\mu^*\mathcal{L}|$ on $A \times A$. By applying the projection formula and by the fact that $\mu$ is a morphism with connected fibers, we have that
\begin{equation*}
H^0(A \times A, \mu^*\mathcal{L}) \cong H^0(A, \mathcal{L}) \PUNKT
\end{equation*}
\noindent Hence, a rational map $\addl: \PR^N \times \PR^N \dashrightarrow \PR^N$ of bidegree $(m,n)$ such that the previous diagram commutes can be expressed as a global section of 
\begin{equation} \label{Mmndefinition}
\mathcal{M}_{m,n} := \mu^*\mathcal{L}^{-1} \otimes \pi_1^* \mathcal{L}^{m} \otimes \pi_2^* \mathcal{L}^{n} \PUNKT
\end{equation}
Thus, our discussion justifies the following definition:
\begin{definition} (Addition law, \cite{LangeRuppert})
Let $(A, \mathcal{L})$ be a polarized abelian variety, where $\mathcal{L}$ il assumed to be very ample. Let $m$, $n$ two non-zero natural numbers. An \textbf{addition law of bidegree $(m,n)$ on $A$} is a global section of $\mathcal{M}_{m,n}$, the latter defined as in \ref{Mmndefinition}. 
\end{definition}
\noindent Let $s \in H^0(A \times A, \mathcal{M}_{m,n})$ be a non-zero addition law of bidegree $(m,n)$. If we consider the rational map $\addl: \PR^N \times \PR^N \dashrightarrow \PR^N$ defined by $s$, and $I_A$ the homogeneous defining ideal of $A$ in $\PR^N$, then the restriction of $\addl$ to $A \times A$ is given by some bihomogeneous polynomials $f_0 \cdots f_{N}$ of bidegree $(m,n)$ in $k[A] = k[\PR^{N}]/I_A$ which define the group law $\mu$ on $A$ away from the base locus $Z$ of $W(\addl)$. The locus $Z$, which is the indeterminacy locus of the rational map $\addl$, will be called \textbf{exceptional locus of $s$}. By looking at the map $\mu \circ \phi$ in diagram \ref{additionlawdiagram}, it can be seen now that this exceptional locus coincides with $div(s)$, and it is, in particular, a divisor in $A \times A$.

\begin{definition}
A set of addition laws $s_1 \cdots s_k$ of bidegree $(m,n)$ is said to be a \textbf{complete set of addition laws} if:
\begin{equation*}
div(s_1) \cap \cdots \cap div(s_k) = \emptyset \PUNKT
\end{equation*}
In particular, there exists a complete set of addition laws of bidegree $(m,n)$ if and only if $|\mathcal{M}_{m,n}|$ is base point free.
\end{definition}
\noindent The problem of determining, whether for a given bidegree $(m,n)$ with $m,n \geq 2$ there exists an addition law (resp. a complete set of addition laws), has been solved by Lange and Ruppert (see \cite{LangeRuppert} p. 610). Their main result is:
\begin{theorem}
Let $A$ be an abelian variety embedded in $\PR^N$, and $\mathcal{L} = \mathcal{M}^{m}$, with $m \geq 3$, a very ample line bundle defining the embedding of $A$ in $\PR^N$. Then:
\begin{itemize} 
\item There are complete systems of addition laws on $A \subseteq \PR^N$ of bidegree $(2,3)$ and $(3,2)$.
\item There exists a system of addition laws on $A \subseteq \PR^N$ of bidegree $(2,2)$ if and only if $\mathcal{L}$ is symmetric. Furthermore, in this case, there exists a complete system of addition laws.
\end{itemize}
\end{theorem}
\noindent We focus now our attention on the case of biquadratic addition laws. 
When the line bundle $\mathcal{L}$ is symmetric, by applying the projection formula (note moreover that $\delta$ is a proper morphism with connected fibers) we have that
\begin{equation*}
H^0(A \times A, \mathcal{M}_{(2,2)}) = H^0(A \times A, \delta^* \mathcal{L}) \cong H^0(A, \mathcal{L})  \PUNKT
\end{equation*}

\noindent We see first a model of a smooth elliptic curve in $\PR^3$ not contained in any hyperplane:
\begin{definition}(Jacobi's model, see also \cite{AdditionLaw} p.21)
Let $u$, $v$, $w$ be three non-zero complex numbers such that $u+v+w = 0$. We denote by $\mathcal{J}_{u,v}$ the elliptic curve in $\PR^3$ with coordinates $X, \cdots, T$ defined as the complete intersection of two of the following three quadrics:
\begin{equation} \label{EquationsJacobi}
\mathcal{J}_{u,v}: \begin{cases}
uX^2 + Y^2 &= Z^2\\
vX^2 + Z^2 &=T^2 \\
wX^2 + T^2 &= Y^2 \PUNKT
\end{cases} 
\end{equation}
\end{definition}
\noindent On $\PR^3 \times \PR^3$, we denote by $[X_0 \cdots T_0]$ the coordinates on the first factor and by $[X_1 \cdots T_1]$ the coordinates for the second one. 
An explicit basis of the space of the biquadratic addition laws has been in determined \cite{AdditionLaw}: 
\begin{theorem} \label{thmADDITIONLAWS}
The vector space $H^0(J_{u,v} \times J_{u,v}, \mathcal{M}_{(2,2)})$ of the addition laws of bidegree $(2,2)$ for the elliptic curve $J_{u,v}$ in $\PR^3$ defined by the Jacobi quadratic equation is generated by:
\begin{align*}
\addl_X := &[X_0^2Y_1^2 - Y_0^2X_1^2, X_0Y_0Z_1T_1 - Z_0T_0X_1Y_1, \\
		&X_0Z_0Y_1T_1 - Y_0T_0X_1Z_1, X_0T_0Y_1Z_1 - Y_0Z_0X_1T_1]\\
\addl_Y := &[X_0Z_0Y_1T_1 + Y_0T_0X_1Z_1, -uX_0T_0X_1T_1 + Y_0Z_0Y_1Z_1, \\
		&uvX_0^2X_1^2 + Z_0^2Z_1^2, vX_0Y_0X_1Y_1 + Z_0T_0Z_1T_1]\\
\addl_Z := &[X_0Y_0Z_1T_1 + Z_0T_0X_1Y_1, uwX_0^2X_1^2 + Y_0^2Y_1^2, \\
		&uX_0T_0X_1T_1 + Y_0Z_0Y_1Z_1, -wX_0Z_0X_1Z_1 + Y_0T_0Y_1T_1]\\
\addl_T := &[u(X_0T_0Y_1Z_1 + Y_0Z_0X_1T_1), u(wX_0Z_0X_1Z_1 + Y_0T_0Y_1T_1), \\
 &u(-vX_0Y_0X_1Y_1 + Z_0T_0Z_1T_1), -vY_0^2Y_1^2 - wZ_0^2Z_1^2] \PUNKT
\end{align*}
Moreover, for every $H \in \{X,Y,Z,T\}$, the exceptional divisor of $\addl_H$ is $\delta^{*}(H)$, where $H$ denotes the corresponding hyperplane divisor $\PR^3$.
\end{theorem}
\begin{proof}
See \cite{AdditionLaw}, p.22
\end{proof}
\begin{remark} \label{rmkTOR}
\noindent Note that, by theorem \ref{thmADDITIONLAWS}, the exceptional divisor of $\addlS$ is $\delta^{*}((X = 0))$ and the divisor $(X=0)$ on the elliptic curve $\mathcal{J}_{u,v}$ is exactly $T_{id} + T_a + T_b + T_{ab}$,
where $a$ and $b$ are generators of $\GrpJ$ acting on the coordinates of $\PR^3$ as in lemma\ref{fundamentallemmagenus9}, and
\begin{align} \label{2pointsexpression}
T_{id} := \begin{bmatrix} 0\\
  					1\\
					1\\
					1
 	 \end{bmatrix} \ \ 
	 T_a := \begin{bmatrix} 0\\
  					1\\
					-1\\
					-1
 	 \end{bmatrix} \ \ 
	 T_b := \begin{bmatrix} 0\\
  					-1\\
					1\\
					-1
 	 \end{bmatrix} \ \ 
	 T_{ab} := \begin{bmatrix} 0\\
  					-1\\
					-1\\
					1
 	 \end{bmatrix} \PUNKT
\end{align}
As the notation suggests, for every element $g$ of $\GrpJ$ the natural action of the point $T_g$ on $J_{u,v} - \Delta_2)$ via $\addlS$ coincides with the action of $g$. Hence $\Delta_2 := \{T_{id}, T_{a}, T_{b}, T_{ab}\}$ is the group of $2$-torsion points on $J_{u,v}$.

It is moreover possible to verify that, according to theorem \ref{thmADDITIONLAWS}, the addition law $\addlS$ is not defined precisely on the union of the four copies of $\mathcal{J}_{u,v}$ in $\mathcal{J}_{u,v} \times \mathcal{J}_{u,v}$ which correspond to the $2$-torsion points of $\mathcal{J}_{u,v}$:
\begin{equation*}
Z := div(\addlS) = \bigcup_{g \in \GrpJ} \{(P,g.P) \ | \ P \in \mathcal{J}_{u,v}\} \subseteq \PR^3 \times  \PR^3 \PUNKT
\end{equation*}
\end{remark}
\begin{definition} \label{JACOBIADD}
\noindent To simplify the notations we will denote the addition law $\addl_X$ on $\mathcal{J}_{u,v}$ simply by $\addl$, and we denote the defining biquadratic polynomials by
\begin{align} \label{canonicalmapasadditionlaw}
\begin{split}
\eta_{12} &:= 			\begin{vmatrix} 									X_0^2	&	X_1^2\\
  																Y_0^2	&	 Y_1^2
 					 \end{vmatrix}
						\\ 
\omega_{45} &:= 		\begin{vmatrix} 									X_0Y_0	&	X_1Y_1\\
  																Z_0T_0	&	 Z_1T_1\\
 					 \end{vmatrix} \\		
\omega_{67} &:= 		\begin{vmatrix} 									X_0Z_0	&	X_1Z_1\\
  																Y_0T_0	&	 Y_1T_1\\
 					 \end{vmatrix} \\			 
\omega_{89} &:= 		\begin{vmatrix} 									X_0T_0	&	X_1T_1\\
  																Y_0Z_0	&	 Y_1Z_1\\
 					 \end{vmatrix}	\PUNKT
\end{split}
\end{align}
\end{definition}


\begin{definition}{(A more general model in $\PR^3$)} \label{EllipticCurveExample} 
For our applications we need a slightly different model of smooth elliptic curve in $\PR^3$. Under the hypothesis that $a$,$b$,$c$ and $d$ are all distinct complex numbers, the curve $\mathcal{E}$ in $\PR^3$ defined by the following couple of quadrics is a smooth elliptic curve:
\begin{equation}\label{starteq}
\mathcal{E} :=  \begin{cases} aX^2 + bY^2 + cZ^2 + dT^2   =  0 \\
							X^2 + Y^2 + Z^2 + T^2 = 0 \ \ \ \ \ \ \ \PUNKT \end{cases}
\end{equation}
We can see now that, up to a choice of signs which represents the action of a $2$-torsion point on $\mathcal{E}$, we can define an addition law which plays the role of the addition law $\addlS$ defined on the Jacobi model in definition \label{JACOBIADD} and Theorem \ref{thmADDITIONLAWS}. The first step is to work out the equations \ref{starteq} to obtain a Jacobi model isomorphic to $\mathcal{E}$ (see equations \ref{EquationsJacobi}). We have

\begin{equation}
\mathcal{E}:  \begin{cases} \frac{a-d}{b-d}X^2  + Y^2 = \frac{c-d}{b-d}Z^2\\
							-\frac{a-c}{b-c}X^2  + \frac{c-d}{b-c}T^2  = Y^2 \ \ \ \ \ \PUNKT \\
							  \end{cases}
\end{equation}
We consider now $\alpha$ and $\beta$ square roots of $\frac{c-d}{b-d}$ and $\frac{c-d}{b-c}$ respectively. By rescaling the coordinates $Z$ and $T$ with $\alpha$ and $\beta$ we see that $\mathcal{E} \cong \mathcal{J}_{\frac{a-d}{b-d},-\frac{a-c}{b-c}}$ and we obtain on $\PR^3 \times \PR^3$ a rational map corresponding to $\addlS$, which represent an addition law of $\mathcal{E}$, up to the choice of the sign of $\alpha$ and $\beta$:
\begin{equation} \label{actionup2points}
\addlS \colon (P,Q) \daschmapsto  \begin{bmatrix} \eta_{12}(P,Q)\\
  						 \alpha \beta \omega_{45}(P,Q)\\
						 \beta \omega_{67}(P,Q)\\
						\alpha \omega_{89}(P,Q)\\
 					 \end{bmatrix} \PUNKT
\end{equation}
Indeed, the rational map defined in \ref{actionup2points} is an addition law up to the action of a $2$-torsion point, according to remark \ref{rmkTOR}. 
This means that this rational map $\addlS$ represents an operation on $\mathcal{E}$ of the following form:
\begin{equation*}
\tilde{\mu}(P,Q) = \mu(T,\mu(P,Q)) = T + P + Q \BEISTRICH
\end{equation*}
where $T$ is a $2$-torsion point on $\mathcal{E}$. 
\end{definition}


\section{The canonical map of the $(1,2,2)$ Theta-divisor and its \\ geometry} \label{geometryCan}

\noindent The sublinear system of $|\omega_{\Surf}|$ generated by the $\GrpJ$-invariant sections $\eta_{12}$, $\eta_{13}$ and $\eta_{23}$ defines the Gauss map $\GAU: \Surf \longrightarrow {\PR^2}^{\vee}$. This map factors through the isogeny $p$ and the Gauss map of $\Theta$, which can be seen as the map which associates to every divisor $p+q$ on $\Du$ the unique line $l$ in $\PR^2 = \PR(H^0(\Du, \omega_{\Du}))$ which cuts on $\Du$ a canonical divisor greater than $u+v$.

\noindent We aim now to describe the behavior of the component of the canonical map of $\Surf$ which is defined by the other three holomorphic sections of the canonical bundle of $\Surf$, which are $\omega_{45}$, $\omega_{67}$ and $\omega_{89}$.
First, we have that the image of the restriction map $H^0(A, \OC_A(\Surf)) \longrightarrow H^0(\Surf, \omega_{\Surf})$ is the subspace generated by $\omega_{45}$, $\omega_{67}$ and $\omega_{89}$.

\begin{definition} \label{importantsectionsS}
In the decomposition in $1$-dimensional vector spaces
\begin{align*}
H^0(A, \OC_A(\Surf)) &= \bigoplus_{\chi \in \GrpJ}H^0(\Jac, \OC_{\Jac}(\Du^{(2)}) \otimes L_{\chi}) \BEISTRICH
\end{align*}
where $L_{\chi}$ are $2$-torsion line bundle on $\Jac$, we have that
\begin{align*}
H^0(\Jac, \OC_{\Jac}(\Du^{(2)}) \otimes L_{\chi_a}) &= \left< \omega_{45} \right> \\
H^0(\Jac, \OC_{\Jac}(\Du^{(2)}) \otimes L_{\chi_b}) &= \left< \omega_{67} \right> \\
H^0(\Jac, \OC_{\Jac}(\Du^{(2)}) \otimes L_{\chi_{ab}}) &= \left< \omega_{89} \right> \BEISTRICH
\end{align*}
where $\chi_g$ denotes the unique non-trivial character of $\GrpJ$ such that $\chi_g(g) = 1$.
Clearly, for every non-trivial element $g$ there exists a unique non-trivial character $\chi$ of $\GrpJ$ whose kernel is generated by $g$, and we can define $\Surf_g$ as the zero locus of the generator of $H^0(\Jac, \OC_{\Jac}(\Theta) \otimes L_{\chi_g})$.

The multiplication by $-1$ in the Jacobian $\Jac$ corresponds to the Serre involution in $\Du^{(2)}$, which sends a divisor $p+q$ to the unique divisor $r+s$ such that $p+q+r+s$ is a canonical divisor on $\Du$. Hence, all global sections of $H^0(A, \Surf)$ are odd, being $\Du^{(2)}$ a translated of $\Theta$ with an odd theta characteristic, and being $\Theta$ the zero locus of the Riemann Theta function, which is an even function.
Moreover, one can easily see that the base locus of $|\OC_A(\Surf)|$ is a set of $16$ points ($A$ is supposed to be general), which on $\Surf$ is defined as the set where $\omega_{45}$, $\omega_{67}$ and $\omega_{89}$ vanish. In remark \ref{examplebasepts} we will characterize this locus in terms of the equation of the curve $\Du$.
\end{definition}



\begin{definition}
\noindent Let $U := [(P,Q)]$ be a point on $\Surf$, and 
$r := \GAU(U)$ the line $\{ax + by + cz = 0\}$, where
\begin{align} \label{definitionsbasicletters}
[a, b,c] &= [\eta_{12}(U), -\eta_{02}(U),\eta_{01}(U) ] \in \PR^2 \PUNKT
\end{align}
The pullback of this line through the rational map $\psi \colon \PR^3 \longrightarrow \PR^3$ which squares the coordinates (see \ref{corollarycover}) is the quadric
\begin{align*}
\mathcal{R}_{U}: aX^2 + bY^2 + cZ^2 = 0 \PUNKT
\end{align*}


\noindent Finally, we denote bt $\mathcal{E}_{U}$ the locus defined by the intersection of $\mathcal{R}_{U}$ with  the $\GrpJ$-invariant quadric of $\PR^3$ containing $\Cu$ (see equation \ref{corollaryscrolleq}):

\begin{equation*}
\mathcal{E}_{U} :=  \begin{cases} aX^2 + bY^2 + cZ^2   =  0 \\
							X^2 + Y^2 + Z^2 + T^2 = 0 \end{cases} \PUNKT
\end{equation*}

The curve $\mathcal{E}_{U}$ is a smooth curve of genus $1$ if and only if $a$, $b$ and $c$ are non zero and all distinct. In this case, (c.f. definition \ref{EllipticCurveExample}) there exist two constants $\alpha_U$ and $\beta_U$, which depend only on $a$, $b$, $c$, and a biquadratic addition law $\addlSP{U}$ on $\mathcal{E}_{U}$, which is defined as follows:
\begin{equation*}
\addlSP{U}: (X, Y) \daschmapsto \begin{bmatrix} \eta_{01}(X,Y)\\
  						 \alpha_U \beta_U \omega_{45}(X,Y)\\
						 \beta_U \omega_{67}(X,Y)\\
						\alpha_U \omega_{89}(X,Y)\\
 					 \end{bmatrix} \PUNKT
\end{equation*}
\end{definition}

By definition it follows that, if for two points $U=[P,Q]$ and $V=[R,S]$ we have that $\phi_{K_S}(U) = \phi_{K_S}(V)$, then $U$ and $V$ define the same locus $\mathcal{E}_{U}$. We prove now that a closer relationship between the group law $\mathcal{E}_{U}$ and the canonical group of $\Surf$ holds:


\begin{lemma} \label{interiorsumlemma}  
Let be $U = [P,Q]$ and $V = [R,S]$ two points of $\mathcal{S}$ such that $\mathcal{E}_U$ and $\mathcal{E}_V$ are smooth. If $\phi_{K_S}(U) = \phi_{K_S}(V)$, then  
$\mathcal{E}_U = \mathcal{E}_V$ and $\mu_U(P,Q) = \mu_U(R,S)$ holds, where $\mu_U$ is the group law in $\mathcal{E}_U$.
\end{lemma}
\begin{proof}
Let us consider the addition law $\addlSP{U}$ defined on $\mathcal{E}_U$. For every point $W=[A,B]$ in a suitable neighborhood $\mathcal{U}$ of $U$ in $\Surf$, the locus $\mathcal{E}_W$ is still a smooth elliptic curve, and we can then denote by $\tau_W$ a corresponding element in the Siegel upper half plane $\mathcal{H}_1$ such that $\mathcal{E}_W = \bigslant{\C}{\Z \oplus \tau_W\Z}$. Moreover, for every $W$ in such a neighborhood it is well-defined $\mu_W(W)$, where $\mu_W$ denotes the group law in $\mathcal{E}_W$ and
\begin{equation*}
\mu_W(W) := \mu_W(A,B) \PUNKT
\end{equation*}
Indeed, it can be easily seen that the definition does not depend on the choice of the representative of $W$. 

\noindent We denote now by $\theta_0(z,\tau_W), \theta_1(z,\tau_W), \theta_2(z,\tau_W), \theta_3(z,\tau_W)$ the four theta functions defining the embedding of $\mathcal{E}_W$ in $\PR^3$, and by $\Psi$ the holomorphic map $\Psi: \mathcal{U} \longrightarrow \PR^3$ defined as follows:
\begin{equation*}
\begin{pmatrix} 1 & & & \\
			  & \alpha & & \\
			  & & \beta & \\
			  & &  & \alpha\beta \\
			 \end{pmatrix} \circ \pi \circ \phi_{\omega_S} \BEISTRICH
\end{equation*}
 where $\pi$ is the following projection $\PR^5 \dashrightarrow \PR^3$:
 \begin{equation*}
 [\eta_{01}, \eta_{02}, \eta_{12}, \omega_{45}, \omega_{67}, \omega_{89}] \daschmapsto [\eta_{01}, \omega_{45}, \omega_{67}, \omega_{89}] \BEISTRICH
 \end{equation*}
 and $\alpha$ and $\beta$ determinations of square roots of $-\frac{\eta_{12}}{\eta_{13}}$ and $-\frac{\eta_{12}}{\eta_{13}+\eta_{12}}$ respectively, which are defined according to definitions \ref{definitionsbasicletters} and \ref{EllipticCurveExample}.
The map $\Psi$ is defined everywhere on $\mathcal{U}$ because, on every point of $\mathcal{U}$, we have that $\eta_{12} \neq 0$ and $\eta_{12} \neq - \eta_{13}$ by definition of $\mathcal{U}$, and in particular $\alpha$ and $\beta$ can be considered simply as holomorphic functions defined on $\mathcal{U}$ as well and with values in $\C^*$. We remark, furthermore, that the choice of the branch of the square root used to define $\alpha$ and $\beta$ is not important because another choice leads to a sign-change of the coordinates to the function $\Psi$ accordingly to the action of the group $\GrpJ$ on the coordinates of $\PR^3$ (cf. definition \ref{EllipticCurveExample}).
The map $\Psi$ is then:
\begin{align*} 
\Psi(W)&= [\eta_{01}(W), \alpha \beta \omega_{45}(W), \beta \omega_{67}(W), \alpha \omega_{89}(W)] = \addlSP{W}(W) \\
&= [\theta_0 \circ \mu_W(W), \theta_1 \circ \mu_W(W), \theta_2 \circ \mu_W(W), \theta_3 \circ \mu_W(W)] \PUNKT
\end{align*}
Hence, if $\phi_{\Surf}(U) = \phi_{\Surf}(V)$ and $\mathcal{E}_U = \mathcal{E}_V$ are smooth elliptic curves, then $\Psi(U) = \Psi(V)$ and in particular there exists a non-zero $\zeta \in \C^*$ such that for every $j = 0, \dotsc 3$ we have
\begin{align*}
\theta_j \circ \mu_U(U) = \zeta \cdot \theta_j \circ \mu_U(V) \PUNKT
\end{align*}
On the other hand, the sections $\theta_j$ on $\mathcal{E}_U$, with $j = 0, \dotsc 3$, embedd $\mathcal{E}_U$ in $\PR^3$, so we can conclude that $\mu_U(U) = \mu_U(V)$.
\end{proof}



\begin{remark}\label{examplebasepts}
In the notation of lemma \ref{fundamentallemmagenus9}, we consider the quartic curve $\Du$ in $\PR^3$ defined by
\begin{align*}
\Du: \begin{cases}
x + y + z + t &= 0 \\
q(x,y,z,t)^2&= xyzt \PUNKT
\end{cases}
\end{align*}
We see that the lines $x$, $y$, $z$ and $t$ in the plane $H: x+y+z+t=0$ are bitangents. For every such a line $l$ we denote by $l_1 + l_2$ the effective divisor on $\Du$ such that
\begin{equation*}
l.\mathcal{D} = 2(l_1 + l_2) \PUNKT
\end{equation*}
We select two points $L_1$ and $L_2$ in the respective preimages in $\Cu$ with respect to $p$. Then, by remark \ref{importantsectionsS}, we see that $\GrpJ.[(L_1, L_2)]$ is a $\GrpJ$-orbit of base points for $\mathcal{L}$ in $A$, since $\omega_{45}$, $\omega_{67}$ and $\omega{89}$ vanish on $[(L_1, L_2)]$. Since the set of base points of a $(1,2,2)$-polarization on a generic abelian variety $A$ is a finite set of $2$-torsion points on $A$ of order $16$, we have determined all base points.


\end{remark}

\begin{theorem} \label{FUNDAMENTALPULLBACK}
Let $U$, $V$ be points on $\Surf$ such that $\phi_{\Surf}(U) = \phi_{\Surf}(V)$. Then one of the following cases occurs:
\begin{itemize}
\item $V = U$
\item $V = -g.U$ for some non-trivial element $g$ of $\GrpJ$. This case arises precisely when $U$ and $V$ belong to the canonical curve $\Surf \cap \Surf_g$.   
\item $V = g.U$  for some non-trivial element $g$ of $\GrpJ$.  This case arises precisely when $U$ and $V$ belong to the translate $\Surf_{h}$, for every $h \in \GrpJ - \{g\}$.
\item $U$ and $V$ are two base points of $|\Surf|$ which belong to the same $\GrpJ$-orbit.
\end{itemize}
\end{theorem}

\begin{proof}
Let us consider $U= [P,Q]$ and $V= [R,S]$ two points on $\Surf$, and let us assume that $\phi_{\Surf}(U) = \phi_{\Surf}(V)$. Let $p$, $q$, $r$ and $s$ denote, moreover, the corresponding points on $\Du$, and $[a,b,c] = [\eta_{23}, -\eta_{13}, \eta_{12}]$ the coefficients of the line $l:= \GAU(U) = \GAU(V) \in {\PR^2}^{\vee}$ according to \ref{definitionsbasicletters}.

\noindent Depending on the coefficients, the locus $\mathcal{E}:= \mathcal{E}_{U}$ will be smooth or not. However, up to exchange $a$, $b$, and $c$ we can assume that we are in one of the following cases:
\begin{itemize}
\item[i)] $a$, $b$ and $c$ are all distinct and non-zero. In this case, $\mathcal{E}$ is a smooth elliptic curve.
\item[ii)] $c = 0$, but $b \neq 0 \neq a$ and $a \neq b$. In this case, the locus $\mathcal{E}$ is the union of two irreducible conics in $\PR^3$ which meet in a point not on $\Cu$. 
\item[iii)] $c = 0$ and $b=0$. In this case, $l$ is the bitangent $x$, and $\mathcal{E}$ is a double conic contained in the hyperplane $\{X=0\}$ in $\PR^3$. This case occurs precisely when $U$ and $V$ are base points. (cf. definition \ref{examplebasepts})
\item[iv)] $c = 0$ and $a = b \neq 0$. In this case, the locus $\mathcal{E}$ is the union of four lines, each couple of them lying on a plane and intersecting in a point not belonging to $\Cu$. 
\end{itemize}
We begin with the first case and we assume that $\mathcal{E}$ is a smooth elliptic curve. Then by lemma \ref{interiorsumlemma}, we have that:
\begin{equation} \label{eqmu}
\mu (P,Q) = \mu (R,S) \BEISTRICH
\end{equation}
where $\mu$ is the group law in $\mathcal{E}$, and we assume that $U \neq V$. Up to exchange $R$ and $S$ we can suppose that $R \neq P$ and $S \neq Q$ by the previous identity \ref{eqmu}.

\noindent If $S$ belongs to the $\GrpJ$-orbit of $P$, we can assume that $S = P$, because we can act on the representatives of $U$ and $V$ with the diagonal subgroup $\Delta_{\GrpJ}$, and by (\ref{eqmu}) it follows that $R = Q$, and finally that $U = V$. Thus, we shall assume that $S$ does not belong to the $\GrpJ$-orbit of $P$, and  that the $\GrpJ$-orbits of $R$ and $S$ are disjoint from the $\GrpJ$-orbits of $P$ and $Q$. Thus, the points of the canonical divisor $p+q+r+s$ are such that $p \neq r$, $p \neq s$, $q \neq r$ and $q \neq s$, and the divisor $R+S$ on $\Cu$ is the preimage of the Serre dual of the divisor $p+q$ on $\Du$. Hence, it must exist an element $g \in G$ such that:
\begin{equation*}
V= -g.U
\end{equation*}
The element $g$ is not the identity because otherwise $U$ and $V$ were both base points (see definition \ref{importantsectionsS}), and in such a case we would reach a contradiction by remark \ref{examplebasepts} since $\mathcal{E}$ cannot be smooth in this case. Hence, the theorem is proved in this case.
\newline
\\
In the remaining cases is not possible to apply lemma \ref{interiorsumlemma}, since $\mathcal{E}$ is no longer  smooth. Nevertheless, we can assume without loss of generality that $U = [P,Q]$ and $V = [P,R]$, where $P,Q,R$ are three points on $\Cu$. 

\noindent Suppose we are in the second case. Then $\mathcal{E}$ is defined by the equations:
\begin{align*}
\mathcal{E} &:=  \begin{cases} aX^2 + bY^2   =  0 \\
					X^2 + Y^2 + Z^2 + T^2 = 0 \end{cases} \BEISTRICH
\end{align*}
where:
\begin{align*}
\mathcal{E} &= \mathcal{Q}^{+} \cup \mathcal{Q}^{-}  \\
\mathcal{Q}^{\epsilon} &= \begin{cases} Y =  \epsilon i \sqrt{\frac{b}{a}} X \\
					X^2 + Y^2 + Z^2 + T^2 = 0 \end{cases}
\end{align*}
and $\epsilon$ denotes a sign. We choose the following parametrization $f^{\epsilon}: \PR^1 \longrightarrow \mathcal{Q}^{\epsilon} \subseteq \PR^3$ : 
\begin{align}\label{secondfunctionexpression}
f^{\epsilon}([u,v]) := \left[\frac{uv}{\sqrt{1-\frac{b}{a}}}, \epsilon i \sqrt{\frac{b}{a}}\frac{uv}{\sqrt{1-\frac{b}{a}}}, \frac{i}{2}(u^2 + v^2),\frac{i}{2}(u^2 - v^2)  \right] 
\end{align} 
\begin{align*}
\mathcal{Q}^{+} \cap \mathcal{Q}^{-} = f^{*}([1,0]) = f^{*}([0,1]) \notin \Cu \PUNKT
\end{align*}
The choice of the square roots in definition \ref{secondfunctionexpression} is not important. Furthermore, we notice that the group $\GrpJ$ acts in the following form:
\begin{align*}
a.f^{\epsilon}([u,v])  &= f^{\epsilon}([u,-v]) \\
b.f^{\epsilon}([u,v])  &= f^{-\epsilon}([v,u]) \PUNKT
\end{align*}
Hence, without loss of generality, we can assume that
\begin{align*}
P &:= f^{1}([u,1]) \\
Q &:= f^{1}([v,1]) \\
R &:= f^{\epsilon}([w,1])  \\
\phi_{\Surf}&(U) = \phi_{\Surf}(V) \PUNKT
\end{align*}
Moreover, without loss of generality we can assume that $R$ does not belong to the $\GrpJ$-orbit of $P$. In this setting, we have to prove that $v = w$ and that $\epsilon = 1$. 
First of all, we have that:
\begin{align*}
\eta_{02}(V) &= \eta_{02}([f^{1}([u,1]), f^{1}([w,1])]) = -\frac{1}{4(1-\frac{a}{b})}(u^2 - w^2)(1-u^2w^2) \\
\eta_{12}(V) &= \eta_{12}([f^{1}([u,1]), f^{1}([w,1])]) = -\frac{1}{4} \frac{\frac{a}{b}}{(1-\frac{a}{b})} (u^2 - w^2)(1-u^2w^2) \PUNKT
\end{align*}
In the same way, the following expressions of the sections $\omega_{45}$, $\omega_{67}$ and $\omega_{89}$ hold, up to a constant independent from $u$, $v$ and $\epsilon$:
\begin{equation} \label{omegaexpressionsecondcase}
\begin{split}
\omega_{45}(V) &= \begin{vmatrix} u^2  & \epsilon w^2 \\
							u^4-1 & w^4-1
							 \end{vmatrix} =
 \begin{cases} -(u^2 - w^2)(u^2w^2 + 1) \ \ \ \text{if} \ \ \epsilon = 1 \\
					 	  (u^2 + w^2)(u^2w^2 - 1) \ \ \ \text{if} \ \ \epsilon = -1 \\
			     \end{cases} \\
\omega_{67}(V) &= \begin{vmatrix} u(u^2 +1)  & \epsilon w(w^2+1) \\
						     u(u^2-1) & w(w^2-1)
							 \end{vmatrix} =
 \begin{cases} -2 uw(u^2 - w^2) \ \ \ \text{if} \ \ \epsilon = 1 \\
					 	  -2 uw(u^2w^2-1)\ \ \ \text{if} \ \ \epsilon = -1 \\
			     \end{cases} \\
\omega_{89}(V) &= \begin{vmatrix} u(u^2 -1)  & \epsilon w(w^2-1) \\
						     u(u^2+1) & w(w^2+1)
							 \end{vmatrix} =
 \begin{cases} 2 uw(u^2 - w^2) \ \ \ \text{if} \ \ \epsilon = 1 \\
					 	  -2 uw(u^2w^2-1)\ \ \ \text{if} \ \ \epsilon = -1 \\
			     \end{cases}
\end{split}
\end{equation}
Finally, by applying the previous expressions \ref{omegaexpressionsecondcase} to $U$, we obtain:
\begin{align*}
\phi_{\Surf}(U) = \begin{bmatrix}
\eta_{01}(U) \\
\eta_{02}(U) \\
\eta_{12}(U) \\
\omega_{45}(U) \\
\omega_{67}(U) \\
\omega_{89}(U)
\end{bmatrix} = \begin{bmatrix}
0\\
-\frac{1}{4(1-\frac{a}{b})}(u^2 - v^2)(1-u^2v^2) \\
 -\frac{1}{4} \frac{\frac{a}{b}}{(1-\frac{a}{b})} (u^2 - v^2)(1-u^2v^2) \\
 -(u^2 - v^2)(u^2v^2 + 1) \\
 -2 uv(u^2 - v^2) \\
2 uv(u^2 - v^2)
\end{bmatrix} =  \begin{bmatrix}
0\\
\frac{1}{4(1-\frac{a}{b})}(1+u^2v^2) \\
 \frac{1}{4} \frac{\frac{a}{b}}{(1-\frac{a}{b})}(1-u^2v^2) \\
 u^2v^2 + 1 \\
 2 uv \\
-2 uv
\end{bmatrix} \PUNKT
\end{align*}
If we had that $\epsilon = -1$ , then we would have:
\begin{align*}
\phi_{\Surf}(V) = \begin{bmatrix}
\eta_{01}(V) \\
\eta_{02}(V) \\
\eta_{12}(V) \\
\omega_{45}(V) \\
\omega_{67}(V) \\
\omega_{89}(V)
\end{bmatrix} = \begin{bmatrix}
0\\
-\frac{1}{4(1-\frac{a}{b})}(u^2 - w^2)(1-u^2w^2) \\
 -\frac{1}{4} \frac{\frac{a}{b}}{(1-\frac{a}{b})} (u^2 - w'^2)(1-u^2w^2) \\
(u^2 + w^2)(u^2w^2 - 1) \\
-2 uw(u^2w^2-1) \\
-2 uw(u^2w^2-1)
\end{bmatrix} =  \begin{bmatrix}
0\\
-\frac{1}{4(1-\frac{a}{b})}(u^2 - w^2)  \\
 -\frac{1}{4} \frac{\frac{a}{b}}{(1-\frac{a}{b})} (u^2 - w^2)  \\
u^2 + w^2\\
-2 uw \\
-2 uw
\end{bmatrix} \BEISTRICH
\end{align*}
which would imply that $\phi_{\Surf}(U) \neq \phi_{\Surf}(V)$ since neither $u$ nor $w$ can vanish. Hence, we can conclude that $\epsilon = \epsilon'= 1$. 
In this case, we have, as points on $\PR^5$:
\begin{align*}
\phi_{\Surf}(U) = \begin{bmatrix}
0\\
\frac{1}{4(1-\frac{a}{b})}(1-u^2v^2) \\
 \frac{1}{4} \frac{\frac{a}{b}}{(1-\frac{a}{b})} (1-u^2v^2) \\
 u^2v^2 + 1 \\
 2 uv \\
-2 uv
\end{bmatrix} =  \begin{bmatrix}
0\\
\frac{1}{4(1-\frac{a}{b})}(1-u^2w^2) \\
 \frac{1}{4} \frac{\frac{a}{b}}{(1-\frac{a}{b})} (1-u^2w^2) \\
 u^2w^2 + 1 \\
 2 uw \\
-2 uw
\end{bmatrix} = \phi_{\Surf}(V)  \BEISTRICH
\end{align*}
and it can be easily seen that $v = w$ holds.

\noindent It only remains to consider the fourth case. The locus $\mathcal{E}$ is reducible and it is the union of four lines,
\begin{equation} \label{linesfourthcase}
\begin{split}
\mathcal{E} &= r^{1,1} \cup r^{1,-1} \cup r^{-1,1} \cup r^{-1,-1}   \\
 r^{\gamma,\delta} &= \begin{cases} Y = \gamma i X  \\
					T = \delta i Z \end{cases} \ \ \ \gamma,\delta \in \{+1, -1\}
\end{split} \PUNKT
\end{equation}
We can now easily parametrize these lines with parametrizations $g^{\gamma, \delta}$, where $g^{\gamma, \delta}([u,v]):= [u, \gamma i u, v, \delta i v]$. 
Denoted by $\infty$ the point $[1,0]$ on the projective line, it can be easily seen that $g^{\gamma, \delta}(0)$ and $g^{\gamma, \delta}(\infty)$ does not belong to $\Cu$, and that the group $\GrpJ$ acts on these lines as follows:
\begin{align*}
a.g^{\gamma, \delta} ([u,v]) &= g^{\gamma, \delta} ([-u,v]) \\
b.g^{\gamma, \delta} ([u,v]) &= g^{-\gamma, -\delta} ([u,v])  \PUNKT
\end{align*}
Let us consider now $U := [g^{\gamma, \delta}(u), g^{\gamma', \delta'}(u') ]$ and $V := [g^{\gamma, \delta}(u), g^{\gamma'', \delta''}(u'') ]$. We assume that their  image with respect to the canonical map is the same. By \ref{polynomialexpressioncanonicalmap}, the evaluation at $U$ of the canonical map $\phi_{\Surf}$ can be expressed as follows: 
\begin{align*}
\phi_{\Surf}(U) = \begin{bmatrix}
\eta_{01}(U) \\
\eta_{02}(U) \\
\eta_{12}(U) \\
\omega_{45}(U) \\
\omega_{67}(U) \\
\omega_{89}(U)
\end{bmatrix}&=
\begin{bmatrix}
0 \\
u^2 - u'^2  \\
u'^2 - u^2  \\
- \begin{vmatrix} \gamma u^2 & \gamma' u'^2 \\
							\delta & \delta'
							 \end{vmatrix} \\
\begin{vmatrix} u & u'  \\
							-\gamma\delta u & -\gamma'\delta' u'
							 \end{vmatrix} \\
\begin{vmatrix} \delta i u &  \delta' i u' v \\
							\gamma i u & \gamma' i u' 
							 \end{vmatrix}
\end{bmatrix} = \begin{bmatrix}
0 \\
u^2 - u'^2  \\
u'^2 - u^2  \\
- \begin{vmatrix} \gamma u^2 & \gamma' u'^2 \\
							\delta & \delta'
							 \end{vmatrix} \\
-\gamma \gamma' uu' \begin{vmatrix} \gamma & \gamma'  \\
				\delta & \delta' 
							 \end{vmatrix} \\
uu' \begin{vmatrix} \gamma & \gamma'  \\
				\delta & \delta' 
							 \end{vmatrix}
\end{bmatrix}
\end{align*}
By the hypothesis that $\phi_{\Surf}(U) = \phi_{\Surf}(V)$, it follows that there exists $\lambda \in \C^*$ such that:
\begin{align} \label{canonicalmapequfin}
\begin{cases}
u^2 - u''^2 &= \lambda (u^2 - u'^2) \\
\begin{vmatrix} \gamma u^2 & \gamma'' u''^2 \\
							\delta & \delta''
							 \end{vmatrix} &= \lambda \begin{vmatrix} \gamma u^2 & \gamma' u'^2 \\
							\delta & \delta'
							 \end{vmatrix} \\
 \gamma'' u'' \Delta'' &= \lambda \gamma' u' \Delta' \\
u'' \Delta'' &= \lambda  u' \Delta' \\
\end{cases}
\end{align}
where $\Delta' := \begin{vmatrix} \gamma & \gamma'  \\
				\delta & \delta' 
							 \end{vmatrix}$ and $\Delta'' := \begin{vmatrix} \gamma & \gamma''  \\
				\delta & \delta'' 
							 \end{vmatrix}$.
In consequence of the last two identities in \ref{canonicalmapequfin}, we can easily infer that $\gamma' = \gamma''$. In particular, we see that $\delta' = \delta''$ because $\Delta'$ vanishes if and only if $\Delta''$ does. Thus, $\Delta' = \Delta''$ and the equations \ref{canonicalmapequfin} can be rewritten in the following form:
\begin{align*} 
\begin{cases}
u^2 - u''^2 &= \lambda (u^2 - u'^2) \\
\begin{vmatrix} \gamma u^2 & \gamma' u''^2 \\
							\delta & \delta'
							 \end{vmatrix} &= \lambda \begin{vmatrix} \gamma u^2 & \gamma' u'^2 \\
							\delta & \delta'
							 \end{vmatrix} \\
u''  &= \lambda  u' \PUNKT
\end{cases}
\end{align*}
We finally obtain the following linear system in the variables $u^2, u'^2$:
\begin{align*}
\begin{cases}
\gamma \delta' (1-\lambda)u^2 &+ \lambda \gamma' \delta (1-\lambda) u'^2 = 0 \\
(1-\lambda)u^2 &+ (1-\lambda) \lambda u'^2 = 0 \PUNKT
\end{cases}
\end{align*}
The determinant of this linear system must vanish because $u$ and $u'$ are supposed to be non-zero. Hence, we have that $\delta\delta'\lambda(1-\lambda)^2\Delta = 0$, which leads to two possible cases: if $\lambda = 1$ we can conclude that $U = V$. Otherwise, $\Delta = 0$ and we have $\omega_{67} =  \omega_{89} = 0$. Hence
\begin{align*}
\begin{cases}
u'' &= \lambda u' \\
u^2 &= -\lambda u'^2 \\
u^2 - u''^2 &= \lambda (u^2-u'^2) \BEISTRICH
\end{cases}
\end{align*}
and finally
\begin{align*}
(-\lambda-1) u'^2 = \lambda (-\lambda u'^2 - u'^2) = -\lambda(\lambda +1)u'^2 \PUNKT
\end{align*}
In conclusion, $\lambda = -1$ and $(\gamma'', \delta'') =  \pm(\gamma', \delta')$, and there exists then a nontrivial element $g$ of $\GrpJ$ such that $g.U = V$. This completes the proof of the theorem.
\end{proof}
In \cite{Cesarano2018} we proved that $\phi_{\Surf}$ has actually injective differential. It is therefore an
interesting question, whether the same result could be proved by using the approach used to prove theorem 
\ref{FUNDAMENTALPULLBACK}.

\newcommand{\BIBND}[4]{\bibitem{#1} 
#2, \emph{#3}. 
#4.}
\newcommand{\BIBA}[5]{\bibitem{#1}
#2, \emph{#3}. 
#4, (#5).}
\newcommand{\BIBB}[6]{\bibitem{#1}
#2, \emph{#3}. 
#4, (#5), #6.}

\end{document}